\documentclass[a4paper,10pt]{amsart}
 
\usepackage[english]{babel}
\usepackage{dsfont} 
\usepackage{graphicx}
\usepackage{color}
\usepackage{enumerate}

\allowdisplaybreaks

\usepackage{amssymb} 
\usepackage{amsthm} 

\newtheorem{thm}{Theorem}
\newtheorem{dfn}[thm]{Definition}
\newtheorem{lem}[thm]{Lemma}

\newtheorem{exm}[thm]{Example}

\newtheorem{rem}[thm]{Remark}
\newtheorem{cor}[thm]{Corollary}

\newcommand{\nset}{\mathds{N}}

\newcommand{\rset}{\mathds{R}}

\newcommand{\pset}{\mathds{P}}
\newcommand{\sset}{\mathds{S}}

\newcommand{\conv}{\mathrm{conv}\,}

\newcommand{\diff}{\mathrm{d}}
\newcommand{\lin}{\mathrm{lin}\,}

\newcommand{\range}{\mathrm{range}\,}

\newcommand{\supp}{\mathrm{supp}\,}

\newcommand{\gdw}{\;\Leftrightarrow\;}
\newcommand{\inter}{\mathrm{int}\,}
\newcommand{\id}{\mathrm{id}}

\newcommand{\cat}{\mathcal{C}}

\newcommand{\cN}{\mathcal{N}}
\newcommand{\cM}{\mathcal{M}}
\newcommand{\cS}{\mathcal{S}}
\newcommand{\cT}{\mathcal{T}}
\newcommand{\cX}{\mathcal{X}}

\newcommand{\cW}{\mathcal{W}}
\newcommand{\cZ}{\mathcal{Z}}

\newcommand{\sA}{\mathsf{A}}
\newcommand{\sB}{\mathsf{B}}

\author{Philipp J.\ di~Dio}
\address{Universit\"at Leipzig, Mathematisches Institut, Augustusplatz 10/11, D-04109 Leipzig, Germany}
\address{Max Planck Institute for Mathematics in the Sciences, Inselstra{\ss}e 22, D-04103 Leipzig, Germany}
\email{didio@uni-leipzig.de}

\begin{title}[The multidimensional truncated Moment Problem: Gaussian\dots]{The multidimensional truncated Moment Problem: Gaussian and Log-Normal Mixtures, their Carath\'eodory Numbers, and Set of Atoms}
\end{title}

\begin{document}

\maketitle

\begin{abstract}
We study truncated moment sequences of distribution mixtures, especially from Gaussian and log-normal distributions and their Carath\'eodory numbers. For $\sA = \{a_1,\dots,a_m\}$ continuous (sufficiently differentiable) functions on $\rset^n$ we give a general upper bound of $m-1$ and a general lower bound of $\left\lceil \frac{2m}{(n+1)(n+2)}\right\rceil$. For polynomials of degree at most $d$ in $n$ variables we find that the number of Gaussian and log-normal mixtures is bounded by the Carath\'eodory numbers in \cite{didio17Cara}. Therefore, for univariate polynomials $\{1,x,\dots,x^d\}$ at most $\left\lceil\frac{d+1}{2}\right\rceil$ distributions are needed. For bivariate polynomials of degree at most $2d-1$ we find that $\frac{3d(d-1)}{2}+1$ Gaussian distributions are sufficient. We also treat polynomial systems with gaps and find, e.g., that for $\{1,x^2,x^3,x^5,x^6\}$ 3 Gaussian distributions are enough for almost all truncated moment sequences. For log-normal distributions the number is bounded by half of the moment number. We give an example of continuous functions where more Gaussian distributions are needed than Dirac delta measures. We show that any inner truncated moment sequence has a mixture which contains any given distribution.
\end{abstract}


\noindent
\textbf{Key  words:} truncated moment problem, Carath\'eodory number, Gaussian, log-normal, finite mixture models, moment method



\section{Introduction}

In many applications, the distribution is a linear combination of simple distributions such as Gaussian distributions
\begin{equation}\label{eq:gaussian}
g_{\xi,\sigma}(x) := \frac{1}{\sqrt{2\pi}\cdot \sigma}\cdot e^{-\frac{(x-\xi)^2}{2\sigma^2}} \quad\text{with}\quad \xi\in \rset,\ \sigma > 0
\end{equation}
or log-normal distributions
\begin{equation}\label{eq:lognormal}
l_{\xi,\sigma}(x) := \begin{cases} \frac{1}{\sqrt{2\pi}\cdot\sigma x} \cdot e^{-\frac{(\log x - \log \xi)^2}{2\sigma^2}} & \text{for}\ x > 0,\\ 0 & \text{for}\ x\leq 0 \end{cases} \quad\text{with}\quad \xi, \sigma\in (0,\infty).
\end{equation}
E.g., in the seminal paper of K.\ Pearson he investigates the distribution of the breadth of the foreheads of Naples Crabs and the length of Carapace of prawns \cite{pearson94}. Since the data did not fit a single Gaussian distribution, he assumed that the distribution comes from a linear combination of two Gaussian distributions
\begin{equation}\label{eq:lincomb2gauss}
\frac{c_1}{\sigma_1 \sqrt{2\pi}}\cdot e^{-\frac{(x-x_1)^2}{2\sigma_1^2}} + \frac{c_2}{\sigma_2 \sqrt{2\pi}}\cdot e^{-\frac{(x-x_2)^2}{2\sigma_2^2}}.
\end{equation}
To determine $c_1$, $c_2$, $\sigma_1$, $\sigma_2$, $x_1$, and $x_2$ he calculated the first five moments of (\ref{eq:lincomb2gauss}) (all are polynomials in $c_1, \dots, x_2$) and after algebraic manipulations got a polynomial of degree 9. The zeros of this polynomial are the solution of fitting (\ref{eq:lincomb2gauss}) to the crab data. This method is now well-known by the name \emph{method of moments}, see e.g.\ \cite{titter85}.

Another frequent distribution is the log-normal distribution (\ref{eq:lognormal}). It appears e.g.\ in the study of option pricings in financial mathematics \cite{stoyan16}, especially in the Black--Scholes model by Black, Scholes \cite{black73}, and Merton \cite{merton73}. In that model it is found that the option pricing is given by
\begin{equation}\label{eq:BlackScholes}
x\cdot g_{0,1}(d_1) - c\cdot e^{r(t-t^*)}\cdot g_{0,1}(d_2)
\end{equation}
with
\[ d_1 = \frac{\log(x/c) + (r + \frac{v^2}{2})(t^*-t)}{v\sqrt{t^*-t}} \quad\text{and}\quad d_2 = \frac{\log(x/c) + (r - \frac{v^2}{2})(t^*-t)}{v\sqrt{t^*-t}},\]
where $t$ is the time variable and $x$, $c$, $t^*$, $r$, and $v$ are parameters of the option/model. Despite the fact that in the Black--Scholes model the linear combination (\ref{eq:BlackScholes}) depends on several parameters and is only related to the log-normal distribution, the log-normal distribution (\ref{eq:lognormal}) is frequently used and one of the most important distributions in financial engineering \cite{stoyan16}.

In the following article we treat the problem of mixtures of densities very general but we also derive more detailed results for the Gaussian (\ref{eq:gaussian}) and log-normal distribution (\ref{eq:lognormal}) because of their importance. We use the following \emph{general setting}:
\begin{enumerate}[(a)]
\item $\delta_{\xi,\sigma}$ are probability measures on a (topological) space $\cX$ with parameters $\xi\in\cX$ and $\sigma\in\Sigma$, $\Sigma$ is the set of parameters (variance; in a larger metric space)

\item $\sA = \{a_1,\dots,a_m\}$ is a set of linearly independent (real valued) continuous functions on the space $\cX$ s.t.\
\[\left|\int_{x\in\cX} a_i(x)~\diff\delta_{\xi,\sigma}(x)\right| < \infty \quad\forall \xi\in\cX,\ \sigma\in\Sigma;\]

\item there exists a $\sigma_0\in\overline{\Sigma}$ (closure of $\Sigma$) such that
\[\lim_{\Sigma\ni\sigma\rightarrow\sigma_0} \int_{x\in\cX} a_i(x)~\diff\delta_{\xi,\sigma}(x) = a_i(\xi) \quad\forall\xi\in\cX,\ i=1,\dots,m\]

\item If the integral $s_i := \int_\cX a_i(x)~\diff\mu(x)$ exists it is called $i$th (or $a_i$-)moment of the measure $\mu$.
\end{enumerate}

The name moment problem comes from $\sA = \{1,x,x^2,\dots,x^d\}$, i.e., the (classical) moments are
$\int_\cX x^i~\diff\mu(x),$
while the general moments are
$\int_\cX a_i(x)~\diff\mu(x)$.
Truncated means that only finitely many moment of $\mu$ are known ($\sA$ is finite). Of course, since the integral is linear in the integrand, the moment problem rather depends on $\lin\sA$ than on $\sA$. So we can always choose an appropriate basis $\sA$ of $\lin\sA$.

\begin{exm}\label{exm:GaussianIntro}
For the Gaussian distributions (\ref{eq:gaussian}) we have $\cX = \rset^n$ ($n\in\nset$), $\Sigma\subset\rset^{n\times n}$ is the set of all symmetric non-singular matrices, $\sigma_0= 0\in\rset^{n\times n}$ is the zero matrix. The \emph{Gaussian measure} $\delta^G_{\xi,\sigma}$ is then defined by
\begin{equation}\label{eq:gaussianMulti}
\diff\delta^G_{\xi,\sigma}(x) := G_{\xi,\sigma}(x)~\diff\lambda^n(x) \quad\text{with}\quad G_{\xi,\sigma}(x) := \frac{\exp\left(-\frac{1}{2}(x-\xi)^T\sigma^{-2}(x-\xi)\right)}{\sqrt{(2\pi)^n \det(\sigma)^2}}
\end{equation}
and $\lambda^n$ is the $n$-dimensional Lebesgue measure. $\sA = \{a_1,\dots,a_m\}\subset C(\rset^n,\rset)$ is a linearly independent set of continuous functions s.t.\ (b) holds. By continuity of the $a_i$'s (b) holds. Then (c) holds, i.e., the Dirac delta measure $\delta_\xi$ is approximated by $\delta^G_{\xi,\sigma}$ if $\sigma\rightarrow\sigma_0 = 0$.
\end{exm}

\begin{exm}\label{exm:LognormalIntro}
Similarly, for the log-normal distribution (\ref{eq:lognormal}) we have $\cX=\rset^n$ (or $\cX = (0,\infty)^n$ with $n\in\nset)$, again $\Sigma\subset\rset^{n\times n}$ the set of all symmetric non-singular matrices, $\sigma_0=0\in\rset^{n\times n}$ the zero matrix. We define the \emph{log-normal measure} $\delta^L_{\xi,\sigma}$ by
\begin{equation}\label{eq:lognormalMulti}
\begin{split}
\diff\delta^L_{\xi,\sigma}(x) &:= L_{\xi,\sigma}(x)~\diff\lambda^n(x) \quad\text{with}\\
L_{\xi,\sigma}(x) &:= \begin{cases} \frac{\exp\left(-\frac{1}{2}(\log x- \log \xi)^T\sigma^{-2}(\log x-\log \xi)\right)}{\sqrt{(2\pi)^n \det(\sigma)^2}\cdot \prod_{i=1}^n x_i} & \text{for}\ x_1,\dots,x_n > 0,\\ 0 & \text{else}\end{cases}
\end{split}
\end{equation}
where $\log x := (\log x_i)_{i=1}^n$ and $\lambda^n$ is again the $n$-dimensional Lebesgue measure. $\sA = \{a_1,\dots,a_m\}\subset C(\cX,\rset)$ is a linearly independent set of continuous functions s.t.\ (b) holds. Then (c) holds, i.e., the Dirac delta measure $\delta_\xi$ is approximated by $\delta^L_{\xi,\sigma}$ if $\sigma\rightarrow\sigma_0=0$.
\end{exm}

For the Gaussian (\ref{eq:gaussian}) and the log-normal distribution (\ref{eq:lognormal}) all moment are known and finite ($n=1$):
\begin{equation}\label{eq:gaussianOneDimMoments}
\int_\rset (x-\xi)^i~\diff\delta^G_{\xi,\sigma}(x) = \begin{cases} (i-1)!!\cdot \sigma^i & \text{for}\ 2|i\\ 0& \text{else}\end{cases}
\end{equation}
and
\begin{equation}\label{eq:logNormalOneDimMoments}
\int_0^\infty x^i~\diff\delta^L_{\xi,\sigma}(x) = \xi^i\cdot e^{\frac{i^2 \sigma^2}{2}}.
\end{equation}
For $n> 1$ similar formulas hold by diagonalizing $\sigma$.

We investigate \emph{mixtures of distributions}
\begin{equation}\label{eq:mixture}
\sum_{i=1}^k c_i\cdot\delta_{\xi_i,\sigma_i} \qquad (c_i > 0)
\end{equation}
with the moment method. In previous works and applications the number $k$ of components is fixed and justified by the model or the data and one of the main questions is the identifiability (uniqueness/determinacy) of (\ref{eq:mixture}), see e.g.\ \cite{pearson94}, \cite{black73}, \cite{titter85}, \cite{martin05}, \cite{permut06}, \cite{stoyan16}, \cite{amendo16}, \cite{amendoToricVarieties1703}, \cite{amendoGaussianMisture1612}, and references therein. But in the present paper we want to investigate the moment cone (Section \ref{sec:momcone}), the possible $\delta_{\xi,\sigma}$ appearing in a representation (\ref{eq:mixture}) (Section \ref{sec:atompos}), and the number $k$ of components needed to represent a given finite number of moments (Section \ref{sec:cat}).

\section{Preliminaries}

The theory and application of moments is rich, see e.g.\ \cite{karlin53}, \cite{richte57}, \cite{rogosi58}, \cite{ahiezer62}, \cite{akhiezClassical}, \cite{kemper68}, \cite{kreinMarkovMomentProblem}, \cite{schmud91}, \cite{matzkePhD}, \cite{reznick92}, \cite{curto3}, \cite{curto2}, \cite{simon98}, \cite{curto00}, \cite{schmud03}, \cite{fialkow05}, \cite{curto05}, \cite{putina08}, \cite{marshallPosPoly}, \cite{lauren09}, \cite{fialkow10}, \cite{curto13}, \cite{lasserreSemiAlgOpt}, \cite{schmud15}, \cite{stoyan16}, \cite{fialkow17}, \cite{infusino17}, \cite{didio17owr}, \cite{schmudMomentBook}, \cite{rienerOptima}, \cite{didio17w+v+}, \cite{didio17Cara}, and references therein. But in the present section we only present definitions and results needed in the following sections, especially from \cite{didio17w+v+} and \cite{didio17Cara} with extensions to mixtures as presented in the introduction.

To efficiently deal with (linear combinations of) Dirac measures $\delta_\xi$ and probability measures $\delta_{\xi,\sigma}$ we introduce the following:

\begin{dfn}
The \emph{moment map} $s_\sA$ is defined by
\[s_\sA: \cX\rightarrow \rset^m,\ x\mapsto s_\sA(x) := \begin{pmatrix}a_1(x)\\ \vdots\\ a_m(x)\end{pmatrix}\]
and for $k\in\nset$ the \emph{moment map} is defined by
\[S_{k,\sA}: \rset_{\geq 0}^k \times \cX^k\rightarrow\rset^m,\ (C,X)\mapsto S_{k,\sA}(C,X) :=\sum_{i=1}^k c_i\cdot s_\sA(x_i)\]
where $C = (c_1,\dots, c_k)$ and $X=(x_1,\dots,x_k)$. We denote by $\cM_\sA$ the set of all (positive) measures $\mu$ on $\cX$ s.t.\ $\left|\int_\cX a_i(x)~\diff\mu(x)\right| < \infty$ for all $i=1,\dots,m$.
\end{dfn}

Clearly, $s_\sA(x)$ is the moment sequence of the Dirac measure $\delta_x$ and $S_{k,\sA}(C,X)$ is the moment sequence of the measure $\mu = \sum_{i=1}^k c_i\cdot \delta_{x_i}$. This and further definitions of course depend on the choice and order of the $a_i$'s in $\sA$. But since the integral is linear in the integrand, reordering or changing the basis $\sA$ does not affect our results. We also write $\mu = (C,X)$ for a finitely atomic measure and we have $\delta_x, (C,X)\in\cM_\sA$. To deal with $\delta_{\xi,\sigma}$ we introduce the following.

\begin{dfn}
We define
\[t_\sA: \cX\times\Sigma\rightarrow \rset^m,\ (x,\sigma)\mapsto t_\sA(x,\sigma) := \left(\int_\cX a_i(y)~\diff\delta_{x,\sigma}(y)\right)_{i=1}^m\]
and
\[T_{k,\sA}: \rset_{\geq 0}^k\times \cX^k\times \Sigma^k\rightarrow \rset^m,\ (C,X,\bar{\sigma})\mapsto T_{k,\sA}(C,X,\bar{\sigma}) := \sum_{i=1}^k c_i\cdot t_\sA(x_i,\sigma_i)\]
where $C=(c_1,\dots,c_k)$, $X=(x_1,\dots,x_k)$, and $\bar{\sigma}=(\sigma_1,\dots,\sigma_k)$.
\end{dfn}

Clearly, $t_\sA(x,\sigma)$ is the moment sequence of $\delta_{x,\sigma}\in\cM_\sA$ and $T_{k,\sA}(C,X,\bar{\sigma})$ is the moment sequence of the mixture $\mu = (C,X,\bar{\sigma}) = \sum_{i=1}^k c_i\cdot \delta_{x_i,\sigma_i}\in\cM_\sA$. From condition (c) we get
\begin{equation}
\lim_{\Sigma\ni\sigma\rightarrow\sigma_0} t_\sA(x,\sigma) = s_\sA(x).
\end{equation}

\begin{dfn}\label{dfn:momentcones}
We define the moment cone
\[\cS_\sA := \left\{ \int_\cX s_\sA(x)~\diff\mu(x) \,\middle|\, \mu\in\cM_\sA\right\} \subseteq\rset^m,\]
its boundary points
\[\partial^*\cS_\sA := \partial\cS_\sA\cap \cS_\sA,\]
and the set
\[\cT_\sA := T_{m,\sA}(\rset_{\geq 0}^m\times\cX^m\times\Sigma^m) = \range T_{m,\sA}.\]
\end{dfn}

$\cT_\sA$ is the set of all moment sequences which have a mixture (\ref{eq:mixture}) as a representing measure with at most $m$ components. It will turn out that $\cT_\sA$ is a convex full-dimensional cone, see Theorem \ref{thm:simpleConeProps}. Of course, $\cT_\sA\subseteq\cS_\sA$ since $(C,X,\bar{\sigma})\in\cM_\sA$ by (b). For the Dirac measures we have the following theorem due to H.\ Richter. See e.g.\ \cite[Thm.\ 1.24]{schmudMomentBook} for a more recent proof.

\begin{thm}[H.\ Richter 1957 {\cite[Satz 4]{richte57}}]\label{thm:richter}
Let $\cX$ be a topological space, $\sA=\{a_1,\dots,a_m\}$ be a finite set of functions on $\cX$, i.e., $\delta_x\in\cM_\sA$ for all $x\in\cX$. Then
\[\cS_\sA = \range S_{m,\sA} = S_{m,\sA}(\rset_{\geq 0}^m\times\cX^m),\]
i.e., for every $\mu\in\cM_\sA$ there is a finitely atomic measure $\mu' = (C,X) = \sum_{i=1}^k c_i\cdot \delta_{x_i}$ with the same moment sequence $\int_\cX a_i(x)~\diff\mu(x) = \int_\cX a_i(x)~\diff\mu'(x)$ and $k\leq m$.
\end{thm}

By the Richter Theorem (Theorem \ref{thm:richter}) every moment sequence $s\in\cS_\sA$ has a finitely atomic representing measure and we can introduce the following number.

\begin{dfn}\label{dfn:caraNo}
Let $s\in\cS_\sA$. We call $\cat_\sA(s)$ defined by
\[\cat_\sA(s) := \min \{k\in\nset \,|\, s\in\range S_{k,\sA}\}\]
the \emph{Carath\'eodory number of $s$}. The \emph{Carath\'eodory number $\cat_\sA$} is
\[\cat_\sA := \max_{s\in\cS_\sA} \cat_\sA(s).\]
\end{dfn}

For the special case of univariate polynomials Richter also proved the following famous result.

\begin{thm}[H.\ Richter 1957 {\cite[Satz 11]{richte57}}]\label{thm:richterOneDim}
Let $\sA = \{1,x,\dots,x^d\}$ on an open, half-open, or closed interval of $\rset$ (or $\cX=\rset$). Then
\[\cat_\sA = \left\lceil \frac{d+1}{2}\right\rceil.\]
\end{thm}

In \cite{didio17Cara} we introduced the following important number.

\begin{dfn}\label{dfn:naDef}
Let $\sA = \{a_1,\dots,a_m\} \subset C^1(U,\rset)$ be a linearly independent subset of $C^1$-functions on an open set $U\subseteq\rset^n$. Define
\begin{equation}\label{eq:NAdef}
\cN_\sA := \min\{ k\in\nset \,|\, DS_{k,\sA}\ \text{has full rank} \}
\end{equation}
where $DS_{k,\sA}$ denotes the total derivative
\begin{equation}\label{eq:DSkAdef}
\begin{split}
DS_{k,\sA} &= (\partial_{c_1} S_{k,\sA},\partial_{x_{1,1}} S_{k,\sA},\dots, \partial_{x_{1,n}} S_{k,\sA}, \partial_{c_2} S_{k,\sA},\dots, \partial_{x_{k,n}} S_{k,\sA})\\
&= (s_\sA(x_1), c_1 \partial_1 s_\sA(x_1),\dots, c_1\partial_n s_\sA(x_1),s_\sA(x_2),\dots, c_k \partial_n s_\sA(x_k))
\end{split}
\end{equation}
of $S_{k,\sA}$.
\end{dfn}

And we proved the following general lower bound on $\cat_\sA$ using Sard's Theorem \cite{sard42}.

\begin{thm}[{\cite[Thm.\ 27]{didio17Cara}}]\label{thm:LowerBoundCara}
Suppose that $\sA = \{a_1,\dots,a_m\} \subset C^r(\cX,\rset)$ be linearly independent with $\cX\subseteq\rset^n$ and $r > \cN_\sA (n+1) -m$. Then
\begin{equation}
\left\lceil \frac{m}{n+1}\right\rceil \leq \cN_\sA \leq \cat_\sA
\end{equation}
and the set of moment sequences $s$ with $\cat_\sA(s)<\cN_\sA$ has $m$-dimensional Lebesgue measure zero in $\rset^m$.
\end{thm}

\begin{rem}\label{rem:manifold}
Instead of $\cX$ being an open subset of $\rset^n$, we could extend the Definition \ref{dfn:naDef} and Theorem \ref{thm:LowerBoundCara} to (differentiable) manifolds $\cX$. By choosing a chart $\varphi: U\subseteq\rset^n\rightarrow\cX$ of the manifold, $U$ open, we have again the previous definition and theorem for $\sA\circ\varphi = \{a_i\circ\varphi \,|\, i=1,\dots,m\}$. It therefore suffices to treat $\cX\subseteq\rset^n$ open or $\cX=\rset^n$.
\end{rem}

For upper bounds we proved an $(m-1)$-Theorem, which we will tighten here.

\begin{thm}[An extension of {\cite[Thm.\ 13]{didio17Cara}}]\label{thm:m-1}
Let $\sA$ and $\cX$ s.t.\ there exists an $e\in\lin\sA$ with $e(x) > 0$ for all $x\in\cX$ and $\range s_\sA\cdot \|s_\sA\|^{-1}$ consists of not more than $m-1$ path-connected components. Then
\[\cat_\sA \leq m-1.\]
\end{thm}
\begin{proof}
The proof is verbatim the same as in {\cite[Thm.\ 13]{didio17Cara}}.
\end{proof}

In \cite{didio17Cara} we missed that we actually only need the assumptions in Theorem \ref{thm:m-1}. We previously stated that $\sA$ must be continuous, there is an $e\in\lin\sA$ s.t.\ $e>0$ on $\cX$ and $\cX$ has not more than $m-1$ components. This of course implies the assumptions in Theorem \ref{thm:m-1}. The key step in the proof was that for any moment sequence $s$ we find by Richter's Theorem (Theorem \ref{thm:richter}) a simplicial cone spanned by $s_\sA(x_1),\dots,s_\sA(x_m)$ containing $s$. Then two $s_\sA(x_i)$ and $s_\sA(x_j)$ lie in the same component of $\range s_\sA\cdot \|s_\sA\|^{-1}$ and can therefore be connected by a path. Following this path shrinks the simplicial cone until $s$ is contained in its boundary, i.e., $s$ needs only $m-1$ atoms.

\section{The Moment Cones $\cS_\sA$ and $\cT_\sA$}
\label{sec:momcone}

In Definition \ref{dfn:momentcones} we defined the moment cones $\cS_\sA$ and $\cT_\sA$ and we already found $\cT_\sA\subseteq \cS_\sA$ and $\cS_\sA$ is a convex cone.

In the following we will ``only'' deal with moment sequences where we know that they have a representing mixture with finitely many components. That is the definition of $\cT_\sA$ in Definition \ref{dfn:momentcones}. However, an application of the Richter Theorem (Theorem \ref{thm:richter}) shows that this is enough.

\begin{dfn}\label{dfn:bi}
Set $\sB := \{b_1,\dots,b_m\}$ where $b_i$ is a function on $\cX\times\Sigma$ defined by
\begin{equation}\label{eq:biDef}
b_i(x,\sigma) := \int_\cX a_i(y)~\diff\delta_{x,\sigma} \qquad\forall (x,\sigma)\in\cX\times\Sigma.
\end{equation}
\end{dfn}

\begin{exm}[Gaussian Distribution, Example \ref{exm:GaussianIntro} revisited]
From (\ref{eq:gaussianMulti}) and (\ref{eq:gaussianOneDimMoments}) we find for $a_i(x) = x^\alpha$ that
$b_i(\xi,\sigma) := \int_{\rset^n} x^\alpha~\diff\delta^G_{\xi,\sigma}(x)$
with $x^\alpha = x_1^{\alpha_1}\cdots x_n^{\alpha_n}$, $\alpha_i\in\nset_0$, is a polynomial in $x_1,\dots,x_n$ and $\sigma_{i,j}$ of degree $|\alpha| = \alpha_1 + \dots + \alpha_n$.
\end{exm}

\begin{exm}[Log-Normal Distribution, Example \ref{exm:LognormalIntro} revisited]
From (\ref{eq:logNormalOneDimMoments}) we find for $a_i(x) = x^i$ on $(0,\infty)$ that
$b_i(\xi,\sigma) := \int_0^\infty x^i~\diff\delta^L_{\xi,\sigma}(x) = \xi^i\cdot e^{\frac{i^2 \sigma^2}{2}}$.
\end{exm}

$\sB$ is well-defined by condition (b). Then any finite or infinite sums of components or continuous versions are measures on $\cX\times\Sigma$. The Richter Theorem for mixtures of distributions reads as follows.

\begin{thm}\label{thm:richterMixtures}
Let $\cX$ and $\Sigma$ be topological spaces, $\sB=\{b_1,\dots,b_m\}$ be a finite set of functions on $\cX\times\Sigma$, i.e., $\delta_{x,\sigma}\in\cM_\sB$ for all $(x,\sigma)\in\cX\times\Sigma$. Then for every $\mu\in\cM_\sB$ there is a mixture with finitely many components $\mu' = (C,X,\bar{\sigma}) = \sum_{i=1}^k c_i\cdot \delta_{x_i,\sigma_i}$ with the same moment sequence and $k\leq m$, i.e.,
\[\left(\int_{\cX\times\Sigma} b_i(x,\sigma)~\diff\mu(x,\sigma)\right)_{i=1}^m \in \cT_\sA = \range T_{m,\sA} = T_{m,\sA}(\rset_{\geq 0}^m\times\cX^m\times\Sigma^m).\]
\end{thm}
\begin{proof}
Apply the Richter Theorem (Theorem \ref{thm:richter}) to $\sB=\{b_1,\dots,b_m\}$ on $\cX\times\Sigma$.
\end{proof}

So it is sufficient to deal ``only'' with moment sequences coming from finite mixtures.

\begin{thm}\label{thm:simpleConeProps}
Let $\sA = \{a_1,\dots,a_m\}$ be linearly independent continuous functions on $\cX$. Then
\begin{enumerate}[i)]
\item $\cT_\sA$ is a full-dimensional convex cone.

\item $\inter\cT_\sA = \inter \cS_\sA$.

\item Assume that
\begin{enumerate}[1)]
\item $\cX$ is a locally compact Hausdorff space,

\item for every $x\in\cX$ and $\sigma\in\Sigma$ there is a compact neighborhood $U_{x,\sigma}\subseteq\supp\delta_{x,\sigma}$ with $\delta_{x,\sigma}(U_{x,\sigma}) > 0$, and

\item for every $f\in\lin\sA$ with $f\geq 0$ on $\cX$ and $f|_U=0$ for a neighborhood $U$ implies $f=0$.
\end{enumerate}
Then
\[\cT_\sA = \inter\cS_\sA\cup\{0\}.\]
\end{enumerate}
\end{thm}
\begin{proof}
i): That $\cT_\sA$ is a cone is clear. That $\cT_\sA$ is convex follows from the Carath\'eodory Theorem for cones, see e.g.\ \cite[Cor.\ 17.1.2]{rock72}. To show that $\cT_\sA$ is full-dimensional, we take $x_1,\dots,x_m\in\cX$ s.t.\ $s_\sA(x_1),\dots,s_\sA(x_m)$ are linearly independent (such $x_i$'s exist since $\sA = \{a_1,\dots,a_m\}$ is linearly independent). Let $(\sigma_i)_{i\in\nset}\subseteq \Sigma$ s.t.\ $\sigma_i\rightarrow \sigma_0$ as $i\rightarrow\infty$. Then
\[\lim_{i\rightarrow \infty} \det(t_\sA(x_1,\sigma_i),\dots,t_\sA(x_m,\sigma_i)) = \det(s_\sA(x_1),\dots,s_\sA(x_m)) \neq 0\]
by condition (c) and continuity of the determinant, i.e., there is an $N\in\nset$ s.t. $\det(t_\sA(x_1,\sigma_N),\dots,t_\sA(x_m,\sigma_N))\neq 0$ and therefore $t_\sA(x_1,\sigma_N),\dots,t_\sA(x_m,\sigma_N)$ are linearly independent in $\rset^m$ and $\cT_\sA$ is full-dimensional.

ii): From $\cT_\sA\subseteq \cS_\sA$ we get $\inter\cT_\sA\subseteq\inter\cS_\sA$. So we have to prove the reverse inclusion $\inter\cT_\sA\supseteq\inter\cS_\sA$. Let $s\in\inter\cS_\sA$. Then there are $x_1,\dots,x_m\in\cX$ and $c_i>0$ s.t.\ $s = \sum_{i=1}^m c_i s_\sA(x_i)$. So $s\in\inter\mathrm{cone}(s_\sA(x_1),\dots,s_\sA(x_m))$. By (c) there exists $\sigma\in\Sigma$ s.t.\ $s\in\inter\mathrm{cone}(t_\sA(x_1,\sigma),\dots,t_\sA(x_m,\sigma))\subseteq\inter\cT_\sA$.

iii): Since $0\in\cT_\sA$ and $\inter\cS_\sA = \inter\cT_\sA \subseteq\cT_\sA$ by ii) we have $\inter\cS_\sA\cup\{0\}\subseteq\cT_\sA$. So it is sufficient to prove the reverse inclusion $\cT_\sA\subseteq\inter\cS_\sA\cup\{0\}$.

Assume this inclusion does not hold, i.e., $\cT_\sA\cap\partial^*\cS_\sA\neq\{0\}$ since $\cT_\sA\subseteq\cS_\sA$. Let $s\in\cT_\sA\cap\partial^*\cS_\sA$, $s\neq 0$, then $\mu = \sum_{i=1}^k c_i t_\sA(x_i,\sigma_i)$ is a non-trivial representing measure of $s$ (since $s\in\cT_\sA$) and there is a $f\geq 0$ in $\lin\sA\setminus\{0\}$ s.t.\ $\int_\cX f(x)~\diff\mu(x) = 0$ (since $s\in\partial^*\cS_\sA$; $f$ is a separating hyperplane supporting $\cS_\sA$ at $s$). Let $U_{x_1,\sigma_1}\subseteq\supp\delta_{x_1,\sigma_1}$ be a compact neighborhood, then by continuity of $f$ and 3) we have $c := \max_{x\in U_{x_1,\sigma_1}} f(x) \in (0,\infty)$. Therefore, $U := U_{x_1,\sigma_1}\cap f^{-1}((c/2,2c))$ is open in $U_{x_1,\sigma_1}$ by continuity of $f$, i.e., $\delta_{x_1,\sigma_1}(U) > 0$. Then
\[0 = \int_\cX f(x)~\diff\mu(x) \geq \int_{U_{x_1,\sigma_1}} f(x)~\diff\delta_{x_1,\sigma_1} \geq \frac{c}{2}\delta_{x_1,\sigma_1}(U) > 0.\]
This is a contradiction, i.e., $\cT_\sA\cap\partial^*\cS_\sA=\{0\}$ and therefore $\cT_\sA\subseteq\inter\cS_\sA\cup\{0\}$.
\end{proof}

In iii) in the previous theorem we actually proved the following. It is a reformulation of Lemma 3 in \cite{didio17w+v+}.

\begin{lem}\label{lem:supportInner}
Assume
\begin{enumerate}[1)]
\item $\cX$ is a locally compact Hausdorff space,

\item for every $x\in\cX$ and $\sigma\in\Sigma$ there is a compact neighborhood $U_{x,\sigma}\subseteq\supp\delta_{x,\sigma}$ with $\delta_{x,\sigma}(U_{x,\sigma}) > 0$, and

\item for every $f\in\lin\sA$ with $f\geq 0$ on $\cX$ and $f|_U=0$ for a neighborhood $U$ implies $f=0$.
\end{enumerate}
Then
\[\mu\ \text{is a representing measure of}\ s\in\cS_\sA\ \text{with}\ \inter\supp\mu\neq\emptyset
\quad\Rightarrow\quad s\in\inter\cS_\sA.\]
\end{lem}

\begin{exm}[Gaussian Mixtures, Example \ref{exm:GaussianIntro} revisited]
For the Gaussian mixtures we have $\cX=\rset^n$ (a locally compact Hausdorff space), $\supp\delta^G_{x,\sigma} = \cX = \rset^n$ for all $x\in\cX=\rset^n$ and $\sigma\in\Sigma\subseteq\rset^{n\times n}$, the set of all symmetric non-singular matrices. Let $\sA$ be a linearly independent set of holomorphic functions (e.g., poly-/monomials). Lemma \ref{lem:supportInner} applies and every moment sequence $s$ is an inner point of the moment cone $\cS_\sA$, i.e., the set of non-zero moment sequences from Gaussian mixtures is open.
\end{exm}

\begin{exm}[Log-Normal Mixtures, Example \ref{exm:LognormalIntro} revisited]
For the Gaussian mixtures we have $\cX=(0,\infty)^n$ (a locally compact Hausdorff space), $\supp\delta^L_{x,\sigma} = \cX = (0,\infty)^n$ for all $x\in\cX=(0,\infty)^n$ and $\sigma\in\Sigma\subseteq\rset^{n\times n}$, the set of all symmetric non-singular matrices. Let $\sA$ be a linearly independent set of holomorphic functions (e.g., poly-/monomials). Lemma \ref{lem:supportInner} applies and every moment sequence $s$ is an inner point of the moment cone $\cS_\sA$, i.e., the set of non-zero moment sequences from log-normal mixtures is open.
\end{exm}

Of course, $\sA$ being holomorphic can be weakened to condition 3) in Lemma \ref{lem:supportInner}.

In Theorem \ref{thm:simpleConeProps}\,ii) we actually showed that any $s\in\inter\cT_\sA$ has a mixture representation with (at most) $m$ components and all components have the same $\sigma$. In the following theorem we will show that we can represent large parts of $\cT_\sA$ by mixture representations with (at most) $m$ components and all components have the same $\sigma$. $\sigma$ must ``just'' be close enough to $\sigma_0$.

\begin{dfn}
%
$\cT_{\sA,\sigma} := T_{m,\sA}(\rset_{\geq 0}^m\times \cX^m\times \{(\sigma,\dots,\sigma)\})$.
\end{dfn}

So $\cT_{\sA,\sigma}$ is the (convex) set of all moment sequences $s$ s.t.\ every $s$ possesses a mixture representation $\sum_{i=1}^k c_i \delta_{x_i,\sigma}$ ($k\leq m$) with at most $m$ components and all components have the same $\sigma$.

\begin{thm}\label{thm:approximateCone}
\begin{enumerate}[i)]
\item $\cT_{\sA,\sigma}$ is a convex cone for all $\sigma\in\Sigma$.

\item Let $(\sigma_i)_{i\in\nset}\subseteq\Sigma$ s.t.\ $\sigma_i\rightarrow\sigma_0$ as $i\rightarrow\infty$.  Then
\[\inter\cT_\sA\cup\{0\} \;\subseteq\; \bigcup_{i\in\nset} \cT_{\sA,\sigma_i} \;\subseteq\; \bigcup_{\sigma\in\Sigma} \cT_{\sA,\sigma}.\]

\item Let $s_1,\dots,s_k\in\inter\cT_\sA$ be points and $(\sigma_i)_{i\in\nset}\subseteq\Sigma$ s.t.\ $\sigma_i\rightarrow\sigma_0$ as $i\rightarrow\infty$. Then there exists an $N\in\nset$ s.t.
\[s_1,\dots,s_k\in\inter\cT_{\sA,\sigma_i}\quad \forall i\geq N.\]

\item Let $K\subset\inter\cT_\sA$ be compact and $(\sigma_i)_{i\in\nset}\subseteq\Sigma$ s.t.\ $\sigma_i\rightarrow\sigma_0$ as $i\rightarrow\infty$. Then there exists an $N\in\nset$ s.t.
\[K\subset\inter\cT_{\sA,\sigma_i} \quad\forall i\geq N.\]

\item Let $C\subset\inter\cT_\sA$ be a closed cone and $(\sigma_i)_{i\in\nset}\subseteq\Sigma$ s.t.\ $\sigma_i\rightarrow\sigma_0$ as $i\rightarrow\infty$. Then there exists an $N\in\nset$ s.t.
\[C\subset\inter\cT_{\sA,\sigma_i} \quad\forall i\geq N.\]

\item Assume conditions 1), 2), and 3) from Theorem \ref{thm:simpleConeProps}\,iii) hold and let $(\sigma_i)_{i\in\nset}\subseteq\Sigma$ s.t.\ $\sigma_i\rightarrow\sigma_0$ as $i\rightarrow\infty$. Then
\[\cT_\sA = \bigcup_{i\in\nset} \cT_{\sA,\sigma_i} = \bigcup_{\sigma\in\Sigma} \cT_{\sA,\sigma}.\]
\end{enumerate}
\end{thm}
\begin{proof}
i): That $\cT_\sA$ is a cone is clear. That $\cT_\sA$ is convex follows from the Carath\'eodory Theorem for cones, see e.g.\ \cite[Cor.\ 17.1.2]{rock72}.

ii): The proof follows the proof of Theorem \ref{thm:simpleConeProps}\,ii). Of course, $0\in\cT_{\sA,\sigma}$ for all $\sigma\in\Sigma$ and the second inclusion holds. So let $s\in\inter\cT_\sA$. Then $s\in\inter\cS_\sA$ by Theorem \ref{thm:simpleConeProps}\,ii) and there are $x_1,\dots,x_m\in\cX$ and $c_i>0$ s.t.\ $s = \sum_{i=1}^m c_i s_\sA(x_i)$. So $s\in\inter\mathrm{cone}(s_\sA(x_1),\dots,s_\sA(x_m))$. By (c) there exists $\sigma_i\in\Sigma$ s.t.
\[s\in\inter\mathrm{cone}(t_\sA(x_1,\sigma),\dots,t_\sA(x_m,\sigma))\subseteq\inter\cT_\sA.\]

iii): As in ii) let $(\sigma_i)_{i\in\nset}\subseteq\Sigma$ s.t.\ $\sigma_i\rightarrow\sigma_0$ as $i\rightarrow\infty$. By ii) for $s_i$ there is an $N_i\in\nset$ s.t.\ $s_i\in\inter\cT_{\sA,\sigma_{l}}$ for all $l\geq N_i$. Set $N := \max\{l_1,\dots,l_k\}$. Then $s_1,\dots,s_k\in\cT_{\sA,\sigma_i}$ for all $i\geq N$.

iv): $\conv K$ is compact since $K$ is compact and $\conv K\subset \conv\inter\cT_\sA = \inter\cT_\sA$ since $\inter\cT_\sA$ is convex. Therefore, $\mathrm{dist}(\partial\cT_\sA,\conv K)>0$ and there are $s_1,\dots,s_k\in\inter\cT_\sA$ s.t.\ $\conv K \subseteq \conv\{s_1,\dots,s_k\}$.
By iii) there is an $N\in\nset$ s.t.\ $s_1,\dots,s_k\in\inter\cT_{\sA,\sigma_i}$ for all $i\geq N$. Since all $\inter\cT_{\sA,\sigma_i}$ are convex, we have $\conv\{s_1,\dots,s_k\}\subset\inter\cT_{\sA,\sigma_i}$ for all $i\geq N$. In conclusion we have
\[K \subseteq \conv K \subseteq \conv\{s_1,\dots,s_k\} \subset\inter\cT_{\sA,\sigma_i} \quad\forall i\geq N.\]

v): Let $\sset^m$ be the unit sphere in $\rset^m$. Then $K = C\cap \sset^m$ is closed and bounded (i.e., compact by the Heine--Borel Theorem) and generates $C$ (i.e., $\mathrm{cone}\, K = C$). By iv) there is an $N\in\nset$ s.t.\ $K\subset\inter\cT_{\sA,\sigma_i}$ for all $i\geq N$. Since $\cT_{\sA,\sigma_i}$ are (convex) cones by i) we have the $C = \mathrm{cone}\, K \subset \mathrm{cone}\,\inter\cT_{\sA,\sigma_i}=\inter\cT_{\sA,\sigma_i}$ for all $i\geq N$.

vi): From Theorem \ref{thm:simpleConeProps}\,ii) and iii) we have $\cT_\sA = \inter\cS_\sA\cup\{0\} = \inter\cT_\sA\cup\{0\}$. Then with ii) in this theorem we have
\[\cT_\sA \;=\; \inter\cT_\sA\cup\{0\} \;\subseteq\; \bigcup_{i\in\nset} \cT_{\sA,\sigma_i} \;\subseteq\; \bigcup_{\sigma\in\Sigma} \cT_{\sA,\sigma} \;\subseteq\; \cT_\sA.\qedhere\]
\end{proof}

\section{Set of Atoms and Identifiability/Uniqueness/Determinacy}
\label{sec:atompos}

We have seen that any $s\in\cT_\sA$ has a finite mixture representation (by definition) and we want to know the possible positions $(x,\sigma)$ s.t.\ $\delta_{x,\sigma}$ appears in any such representation.

\begin{dfn}
Let $s\in\cT_\sA$. The set of atoms (components) $\cW(s)$ is defined by
\[\cW(s) := \left\{(x,\sigma) \,\middle|\, \exists c_i,c>0, (x_i,\sigma_i)\in\cX\times\Sigma: s = c\cdot t_\sA(x,\sigma) + \sum_{i=1}^k c_i t_\sA(x_i,\sigma)\right\}.\]
\end{dfn}

So $\cW(s)$ is the set of all $(x,\sigma)\in\cX\times\Sigma$ s.t.\ $\delta_{x,\sigma}$ appears in a mixture representation of $s$.

\begin{dfn}
Let $s\in\cT_\sA$. $s$ is called \emph{determined} if it has only one mixture representation. Otherwise $s$ is called \emph{indeterminate}.
\end{dfn}

The following theorem summarizes properties of $\cW(s)$ and determinacy. It is a reformulation of Theorems 16 and 19 in \cite{didio17w+v+}.

\begin{thm}\label{thm:atomspos}
Let $s\in\cT_\sA$.
\begin{enumerate}[i)]
\item $s\in\inter\cT_\sA \gdw \cW(s) = \cX\times\Sigma$.
\item $s$ is indeterminate  $\gdw\{t_\sA(x,\sigma) \,|\, (x,\sigma)\in\cW(s)\}$ is linearly dependent.
\item Assume that
\begin{enumerate}[1)]
\item $\cX$ is a locally compact Hausdorff space,

\item for every $x\in\cX$ and $\sigma\in\Sigma$ there is a compact neighborhood $U_{x,\sigma}\subseteq\supp\delta_{x,\sigma}$ with $\delta_{x,\sigma}(U_{x,\sigma}) > 0$, and

\item for every $f\in\lin\sA$ with $f\geq 0$ on $\cX$ and $f|_U=0$ for a neighborhood $U$ implies $f=0$.
\end{enumerate}
Then every $s\in\cT_\sA\setminus\{0\}$ is indeterminate and $\cW(s)=\cX\times\Sigma$ for all $s\in\cT_\sA\setminus\{0\}$.
\end{enumerate}
\end{thm}
\begin{proof}
i) ``$\Rightarrow$'': Let $(x,\sigma)\in\cX\times\Sigma$. Since $s\in\inter\cT_\sA$ there is an $\varepsilon > 0$ s.t.\ $s' := s - \varepsilon\cdot t_\sA(x,\sigma)\in\inter\cT_\sA$. Then $s'$ has a finite mixture representation $\mu'$ and $\mu := \mu' + \varepsilon\cdot \delta_{x,\sigma}$ is a mixture presentation of $s$ containing $\delta_{x,\sigma}$.\\
i) ``$\Leftarrow$'': Let $(x_i,\sigma_i)\in\cW(s)=\cX\times\Sigma$ ($i=1,\dots,m$) s.t.\ $(t_\sA(x_1,\sigma_1),\dots,t_\sA(x_m,\sigma_m))$ has full rank. Let $\mu_i$ be representing mixtures of $s$ s.t.\ every $\mu_i$ contains the component $\delta_{x_i,\sigma_i}$. Then
\[\mu := \frac{1}{m}\sum_{i=1}^m \mu_i = \sum_{i=1}^m c_i \delta_{x_i,\sigma_i} + \sum_{j=1}^k d_j \delta_{y_j,\sigma'_j}\]
is a representing mixture of $s$ which contains all $\delta_{x_i,\sigma_i}$. Then the map
\[S(\gamma_1,\dots,\gamma_m) := \sum_{i=1}^m \gamma_i t_\sA(x_i,\sigma_i) + \sum_{j=1}^k d_j t_\sA(y_j,\sigma'_j)\]
maps a neighborhood $B_\varepsilon((c_1,\dots,c_m))\subset (0,\infty)^n$ with $0< \varepsilon < \min\{c_1,\dots,c_m\}$ to a neighborhood of $s$ since the $t_\sA(x_i,\sigma_i)$ are linearly independent, i.e., $s\in\inter\cT_\sA$.

ii) Apply (i) $\Leftrightarrow$ (ii) in \cite[Thm.\ 19]{didio17w+v+}.

iii): By Theorem \ref{thm:simpleConeProps}\,iii) $\cT_\sA\setminus\{0\} = \inter\cS_\sA$ is open and i) and ii) apply.
\end{proof}

\begin{exm}[Examples \ref{exm:GaussianIntro} and \ref{exm:LognormalIntro} revisited]
For the Gaussian and log-normal distributions the conditions 1), 2), and 3) are fulfilled, i.e., every moment sequence is indeterminate and any component $\delta^G_{x,\sigma}$ or $\delta^L_{x,\sigma}$, respectively, can appear in a representing mixture.
\end{exm}

That for any $s\in\inter\cT_\sA$ any $\delta_{x,\sigma}$ appears in a representing mixture is only possible since the number of components is not restricted. So we need to learn more about the number of components.

\section{The Carath\'eodory Number $\cat^M_\sA$ for Mixtures of Distributions}
\label{sec:cat}

Since (by definition or Theorem \ref{thm:richterMixtures}) every $s\in\cT_\sA$ has a mixture representation we can define the Carath\'eodory number of mixtures similar to Definition \ref{dfn:caraNo}.

\begin{dfn}
Let $s\in\cT_\sA$. Define the \emph{Carath\'eodory number $\cat^M_\sA(s)$ of mixtures of $s$} by
\[\cat^M_\sA(s)\ :=\ \min\ \{k\in\nset \,|\, s\ \text{has a mixture representation with}\ k\ \text{components}\}.\]
The \emph{Carath\'eodory number $\cat_\sA^M$ of mixtures} is defined by
\[\cat_\sA^M\ :=\ \max_{s\in\cT_\sA}\ \cat_\sA^M(s).\]
\end{dfn}

$\cat_\sA^M(s)$ and $\cat_\sA^M$ are well-defined by Theorem \ref{thm:richter} or equivalently Theorem \ref{thm:richterMixtures} since $0\leq \cat_\sA^M(s) \leq \cat_\sA^M \leq m$ and $\cat_\sA(s),\cat_\sA\in\nset_0$. We will shortly see in Example \ref{exm:catMbiggerCat} that not necessarily $\cat_\sA^M \leq \cat_\sA$ even though $\cT_\sA\subseteq\cS_\sA = \range S_{\cat_\sA,\sA}$ holds.

In important cases, e.g., Gaussian and log-normal mixtures (Examples \ref{exm:GaussianIntro} and \ref{exm:LognormalIntro}), the moment cone has no boundary points despite $0$. So ``standard'' methods to bound $\cat_\sA^M$ in \cite{didio17Cara} and \cite{rienerOptima} can not be applied. These ``standard'' methods are, e.g., ``taking an inner point, removing an atom to get to the boundary and describe the boundary'' or ``close the moment cone by going from $\rset^n$ to projective space $\pset^n$ and to homogeneous polynomials''. In all these cases, a boundary point $s\neq 0$ would imply that there is an $f\in\lin\sA$, $f\geq 0$, such that $\supp\mu \subseteq\cW(s)\subseteq \cZ(f) := \{x\in\cX \,|\, f(x)=0\}\neq\cX$ for all representing measures $\mu$ of $s$. But this is not possible as long as conditions 1), 2), and 3) shall hold and $\inter \supp \delta_{x,\sigma} \neq \emptyset$, see Theorem \ref{thm:atomspos}\,iii). So recent methods in \cite{didio17Cara} and \cite{rienerOptima} do not apply. Theorem \ref{thm:approximateCone} fills the gap. But let us start with the lower bounds on $\cat_\sA^M$.

\begin{dfn}
Let $\cX\subseteq\rset^{n_1}$ and $\Sigma\subseteq\rset^{n_2}$ be open. Furthermore, let $b_i$ from Definition \ref{dfn:bi} be $C^1$-functions. We define
\[\cN_\sA^M := \min\{ k\in\nset \,|\, DT_{k,\sA}\ \text{has full rank}\}.\]
\end{dfn}

$DT_{k,\sA}$ is the total derivative of $T_{k,\sA}$.

\begin{exm}\label{exm:GaussianLogNormalXSigma}
For the Gaussian distribution we have $\cX= \rset^n$ and for the log-normal distribution we have $\cX=(0,\infty)^n$, see Examples \ref{exm:GaussianIntro} and \ref{exm:LognormalIntro}. In both cases $\Sigma$ is the set of all symmetric non-singular matrices in $\rset^{n\times n}$, e.g., $\Sigma$ is a open $\frac{n(n+1)}{2}$-dimensional smooth manifold, i.e., Remark \ref{rem:manifold} applies.
\end{exm}

We have the following lower bound on $\cat_\sA^M$.

\begin{thm}\label{thm:lowerBoundMixture}
Let $\cX\subseteq\rset^{n_1}$ and $\Sigma\subseteq\rset^{n_2}$ be open. Furthermore, let $b_i$ from Definition \ref{dfn:bi} be $C^r$-functions with $r > \cN_\sA (n_1+n_2+1) -m$. Then
\[\left\lceil\frac{m}{n_1+n_2}\right\rceil \leq\ \cN_\sA^M\ \leq\ \cat_\sA^M.\]
\end{thm}
\begin{proof}
Apply Theorem \ref{thm:LowerBoundCara} \cite[Thm.\ 27]{didio17Cara} with $\cX\times\Sigma\subseteq\rset^{n_1+n_2}$.
\end{proof}

See also Remark \ref{rem:manifold} for extensions of $\cX$ and $\Sigma$ to differentiable manifolds. The previous theorem then implies that there are cases where $\cat_\sA^M \not\leq\cat_\sA$.

\begin{exm}[$\cat_\sA^M \not\leq\cat_\sA$, see {\cite[Exm.\ 16 and Rem.\ 17]{didio17Cara}}] \label{exm:catMbiggerCat}
Let $\varphi=(\varphi_i)_{i=1}^m$ be the coordinate functions of a space filling curve \cite[Ch.\ 5 and 7]{saganSpaceFillingCurves}, i.e., $\varphi: [0,1]\rightarrow [0,1]^m$ are continuous. Extend all $\varphi_i$ continuously to $\rset$ s.t.\ $\supp\varphi_i\subseteq [-1,2]$. For the Gaussian distributions (Example \ref{exm:GaussianIntro}) we can then interchange differentiation and integration in (\ref{eq:biDef}) in Definition \ref{dfn:bi} by applying a result of Lebesgue (see e.g.\ \cite[Lem.\ 2.8]{grubbDistributions}) and we get that all $b_i$'s are $C^\infty$. Therefore Theorem \ref{thm:lowerBoundMixture} holds and we get with Example \ref{exm:GaussianLogNormalXSigma} that $n+1 + \frac{n(n+1)}{2}= \frac{(n+2)(n+1)}{2}$ and
$\left\lceil \frac{2m}{(n+2)(n+1)}\right\rceil \leq \cat_\sA^M$,
and for $2m > (n+2)(n+1)$ we have $\cat_\sA = 1 < \cat_\sA^M$.
\end{exm}

If we are interested in representations with fixed $\sigma$ for all distributions, we need at least $\left\lceil\frac{m}{n_1+1}\right\rceil$ distributions $\delta_{x_i,\sigma}$. And when we want a presentation s.t.\ all $\sigma_1=\dots=\sigma_k=\sigma$ are the same but we are allowed to chose $\sigma$ freely, then we need at least $\left\lceil\frac{m-n_2}{n_1+1}\right\rceil$ distributions. Apply Theorem \ref{thm:LowerBoundCara} or modify the proof in \cite[Thm.\ 27]{didio17Cara} to prove these.

Let us now treat the upper bound estimates. We already established $\cat_\sA^M \leq m$ in Theorem \ref{thm:richterMixtures}. We can tighten this.

\begin{thm}
Let $\sA$, $\cX$, $\Sigma$, $(\sigma_i)_{i\in\nset}\subseteq\Sigma$, and $\delta_{x,\sigma}$ s.t.\ $\sigma_i\rightarrow\sigma_0$, there exists an $e\in\lin\{b_1,\dots,b_m\}$ from Definition \ref{dfn:bi} and $N\in\nset$ with $e(x,\sigma_i) > 0$ for all $x\in\cX$ and $i\geq N$, and $\range t_\sA\cdot \|t_\sA\|^{-1}$ consists of not more than $m-1$ path-connected components. Then
\[\cat_\sA^M \leq m-1.\]
\end{thm}
\begin{proof}
Let $s\in\cT_\sA$. Then by Theorem \ref{thm:approximateCone}\,ii) there is an $N'\in\nset$ s.t.\ $s\in\cT_{\sA,\sigma_i}$ for all $i\geq N'$.  Apply Theorem \ref{thm:m-1} to $\cT_{\sA,\sigma_i}$ for some $i\geq\max\{N,N'\}$.
\end{proof}

Theorem \ref{thm:approximateCone} can be used to bound $\cat_\sA^M$.

\begin{thm}\label{thm:approxBoundsCara}
Let $(\sigma_i)_{i\in\nset}\subseteq\Sigma$ s.t.\ $\sigma_i\rightarrow\sigma_0$ as $i\rightarrow\infty$ and there exist $C,N\in\nset$ s.t.\ $\cat_{\{b_1(\,\cdot\,,\sigma_i),\dots,b_m(\,\cdot\,,\sigma_i)\}} \leq C$ for all $i\geq N$. Then
\[\cat_\sA^M \leq C.\]
\end{thm}
\begin{proof}
Let $s\in\cT_\sA$. Then by Theorem \ref{thm:approximateCone}\,ii) there is an $N'\in\nset$ s.t.\ $s\in\cT_{\sA,\sigma_i}$ for all $i\geq\max\{N, N'\}$. I.e., $\cat_\sA^M(s) \leq C$ since $\cat_{\{b_1(\,\cdot\,,\sigma_i),\dots,b_m(\,\cdot\,,\sigma_i)\}} \leq C$ for all $i\geq N$. Since $s$ was arbitrary, we have $\cat_\sA^M \leq C$.
\end{proof}

Let us give an application to the most common cases: Gaussian and log-normal distributions (Examples \ref{exm:GaussianIntro} and \ref{exm:LognormalIntro}). Let us start with the following remark.

\begin{rem}\label{rem:multiindex}
Let
$\sA_{n,d} := \{x^\alpha \,|\, \alpha\in\nset_0^n \wedge |\alpha| \leq d\}$
be the monomials of degree at most $d$ in $n$ variables and $(\sigma_i)_{i\in\nset}\subset\Sigma$ with $\sigma_i := i^{-1} \id$, $\id$ the identity matrix. For the Gaussian distributions $\delta_{x,\sigma}^G$ we find from (\ref{eq:gaussianOneDimMoments}) that
$b_\alpha(x,i^{-1}\id) := \int_{\rset^n} y^\alpha ~\diff\delta^G_{x,i^{-1}\id}(y)$
is a polynomial in $x_1,\dots,x_n$ with leading term $x^\alpha$. So
\begin{equation}\label{eq:multiindex}
\lin \{b_\alpha(x,i^{-1}\id) \,|\, \alpha\in\nset_0^n \wedge |\alpha| \leq d\} = \lin\sA_{n,d}.
\end{equation}
Since the Carat\'eodory number does not depend on the choice of basis functions spanning $\lin \{b_\alpha(x,i^{-1}\id)\}$ we can apply Theorem \ref{thm:approxBoundsCara} with results from previous studies of Carath\'eodory numbers from Dirac measures, see e.g.\ \cite{didio17Cara} and \cite{rienerOptima}.

For the log-normal distribution $\delta_{x,\sigma}^L$ we find the same: (\ref{eq:multiindex}) holds by (\ref{eq:logNormalOneDimMoments}). But we have $\cX = (0,\infty)$.
\end{rem}

Let us apply the previous remark.

\begin{thm}\label{thm:generalDiracBoundonGaussians}
Let
$\sA_{n,d} = \{x^\alpha \,|\, \alpha\in\nset_0^n \wedge |\alpha| \leq d\}$
be the monomials of degree at most $d$ in $n$ variables. Then for Gaussian and log-normal mixtures we have
\[\cat_{\sA_{n,d}}^M \leq \cat_{\sA_{n,d}}.\]
\end{thm}
\begin{proof}
Follows from (\ref{eq:multiindex}) and Theorem \ref{thm:approxBoundsCara}.
\end{proof}

Let us give some explicit applications of the previous theorem. For the one-dimensional Gaussian mixture we have

\begin{cor}\label{thm:GaussianCaraExampleOneDim}
Let $\sA = \{1,x,\dots,x^d\}$ on $\rset$, $d\in\nset$. For the Gaussian mixtures we have
\[\left\lceil\frac{d+1}{3}\right\rceil\quad \leq\quad \cat_\sA^M\quad \leq\quad \left\lceil\frac{d+1}{2}\right\rceil,\]
and every moment sequence $s$ coming from a linear combination of Gaussian measures can be written as
\[s = \sum_{i=1}^k c_i s_{\sA,\sigma}(x_i) \quad\text{with}\quad k\leq \left\lceil\frac{d+1}{2}\right\rceil \quad\text{and}\quad \text{some}\quad \sigma = \sigma(s) > 0.\]
Equivalently, every moment sequence $s$ from a Gaussian mixture has a Gaussian mixture representation
\[F(x) = \sum_{i=1}^k c_i e^{-\frac{(x-x_i)^2}{2\sigma^2}} \quad\text{with}\quad k\leq \left\lceil\frac{d+1}{2}\right\rceil \quad\text{and}\quad \text{some}\quad \sigma = \sigma(s) > 0.\]
\end{cor}
\begin{proof}
$\cat_\sA^M \geq \left\lceil\frac{d+1}{3}\right\rceil$ follows from Theorem \ref{thm:lowerBoundMixture} with $n=1$ and the upper bound follows from Theorem \ref{thm:generalDiracBoundonGaussians} with Theorem \ref{thm:richterOneDim}.
\end{proof}

For the one-dimensional log-normal distribution we will even have a more general result since it only lives on $(0,\infty)$, see Theorem \ref{thm:logNormalGapsCaraBound}.

For systems $\sA \subset\rset[x_1,\dots,x_n]$ with gaps, the application of previous results is more involved. (\ref{eq:multiindex}) no longer holds. E.g.\ for $\sA = \{1,x^2,x^3,x^5,x^6\}$ on $\rset$ we get
%
$b_0(x,\sigma) = 1$,
$b_2(x,\sigma) = x^2 + \sigma^2$,
$b_3(x,\sigma) = x^3 + 3\sigma^2 x$,
$b_5(x,\sigma) = x^5 + 10\sigma^2 x^3 + 15\sigma^4 x$,
$b_6(x,\sigma) = x^6 + 15\sigma^2 x^4 + 45\sigma^4 x^2 + 15\sigma^6$,
%
so
\[\lin \{b_i(x,\sigma)\}_{i=1}^6 = \lin\{1,\;\; x^2,\;\; x^3 + 3\sigma^2 x,\;\; x^5 - 15\sigma^4 x,\;\; x^6 + 15\sigma^2 x^4\},\]
i.e., we always have contributions from $x$ and $x^4$.

Systems with gaps, especially the univariate case, were treated in \cite{didio17Cara}. For $\sA = \{1,x^2,x^3,x^5,x^6\}$ on $\rset$ we found that $\cat_\sA = 3$ \cite[Exm.\ 46]{didio17Cara}. Theorem \ref{thm:richterMixtures} gives $\cat_\sA^M\leq 5$ while Theorem \ref{thm:m-1} gives a bound of $\cat_\sA^M \leq 4$. We will show with the following results, at least $\cat_\sA^M(s)\leq 3$ for almost every $s\in\cT_\sA$, see Example \ref{exm:uniGaps}. At first we will show that a $k$-atomic Dirac measure $(C,X)$ s.t.\ $DS_{k,\sA}(C,X)$ has full rank gives an mixture with at most $k$ components.

\begin{thm}\label{thm:GaussianMixtureFromRegMeasure}
Let $\sA\in C^1$ s.t.\ $b_i(x,\sigma\id)$ and $\partial_j b_i(x,\sigma\id)$ are continuous in $\sigma\in[0,\infty)$ and $x\in\rset^n$ for all $i=1,\dots,m$ and $j=1,\dots,n$, and let $s\in\cS_\sA$ s.t.\ $s$ has a $k$-atomic representing measure $(C_s,X_s)$ with $DS_{k,\sA}(C_s,X_s)$ has full rank. Then
\[\cat_\sA^M(s) \leq k.\]
\end{thm}
\begin{proof}
Since $s$ has a $k$-atomic representing measure $(C_s,X_s)$ s.t.\ $DS_{k,\sA}(C_s,X_s)$ has full rank, $s\in\inter\cS_\sA \subset \cT_\sA$ by Theorem \ref{thm:simpleConeProps}. Since $DS_{k,\sA}(C_s,X_s)$ has full rank, pick $m$ variables $y=(y_1,\dots,y_m)$ from $c_1,\dots,c_k$ and $x_{1,1},\dots,x_{k,n}$ s.t.\ $D_y S_{k,\sA}(C_s,X_s)\in \rset^{m\times m}$ is non-singular. Since $b_i(x,\sigma\id)$ and $\partial_j b_i(x,\sigma\id)$ are continuous in $\sigma\in[0,\infty)$ and $x\in\rset^n$ for all $i=1,\dots,m$ and $j=1,\dots,n$, $D_y T_{k,\sA}(C,X,\sigma(\id,\dots,\id)) = D_y S_{k,\{b_i(\,\cdot\,,\sigma\id)\}_i}(C,X)$ is continuous in $\sigma\in[0,\infty)$, $c_i\in [0,\infty)$, and $x\in\rset^n$. Therefore, there is an $\varepsilon > 0$ s.t.\ $D_y S_{k,\{b_i(\,\cdot\,,\sigma\id)\}_i}$ is non-singular for all $(\sigma,C,X)\in K_\varepsilon := [0,\varepsilon]\times \overline{B_\varepsilon(C_s,X_s)}$ since the determinant is continuous in the entries of the matrix. Denote by $\tau_1(\sigma,C,X) \leq \dots \leq \tau_n(\sigma,C,X)$ the singular values of $D_y S_{k,\{b_i(\,\cdot\,,\sigma\id)\}_i}(C,X)$. Since the singular values also depend continuously on the matrix, they depend continuously on $(\sigma,C,X)\in K_\varepsilon$. Since $\tau_i(\sigma,C,X)$ are continuous, they are bounded from above on the compact set $K_\varepsilon$. But since
$\det(D_y S_{k,\sA_{\sigma\id}}(C,X)) = \pm\tau_1(\sigma,C,X) \cdots \tau_n(\sigma,C,X) \neq 0$,
$\tau_1$ is non-zero on $K_\varepsilon$ and there is a $\tau_{\min} > 0$ s.t.
\[\inf_{(\sigma,C,X)\in K_\varepsilon} \tau_1(\sigma,C,X) = \min_{(\sigma,C,X)\in K_\varepsilon} \tau_1(\sigma,C,X) \geq \tau_{\min}.\]
But this means that
$B_r(S_{k,\{b_i(\,\cdot\,,\sigma\id)\}_i}(C_s,X_s)) \subseteq S_{k,\{b_i(\,\cdot\,,\sigma\id)\}_i}(B_\varepsilon(C_s,X_s))$ for all $r\in (0,\min\{\varepsilon,\tau_{\min}\varepsilon\})\ \wedge\ \sigma\in [0,\varepsilon]$.
Fix $r\in (0,\min\{\varepsilon,\tau_{\min}\varepsilon\})$. Then $s\in B_r(S_{k,\{b_i(\,\cdot\,,\sigma\id)\}_i}(C_s,X_s)) \subseteq S_{k,\{b_i(\,\cdot\,,\sigma\id)\}_i}(B_\varepsilon(C_s,X_s))$ for all $\sigma\in (0,\varepsilon)$ s.t.\ $\|s - S_{k,\{b_i(\,\cdot\,,\sigma\id)\}_i}(C_s,X_s)\| < r$, i.e.,
$s = S_{k,\{b_i(\,\cdot\,,\sigma\id)\}_i}(C,X) = T_{k,\sA}(C,X,(\sigma\id,\dots,\sigma\id))$
for a $(C,X)\in B_\varepsilon(C_s,X_s)$.
\end{proof}

Note, that in the proof of Theorem \ref{thm:GaussianMixtureFromRegMeasure} the use of the multiple of $\id$ is arbitrary, any non-singular symmetric matrix will do, just insert a basis transformation on $\rset^n$. From Theorem \ref{thm:GaussianMixtureFromRegMeasure} we get the following.

\begin{thm}\label{thm:DenseGaussianMixture}
Let $\sA = \{a_1,\dots,a_m\}$ in $C^r(\rset^n,\rset)$ with $r > \cN_\sA\cdot (n+1) - m$ s.t.\ $b_i(x,\sigma\id)$ and $\partial_j b_i(x,\sigma\id)$ are continuous in $\sigma\in[0,\infty)$ and $x\in\rset^n$ for all $i=1,\dots,m$ and $j=1,\dots,n$. Then
\begin{equation}\label{eq:gaussianbound}
\cat_\sA^M(s) \quad\leq\quad \cat_\sA(s) \quad\leq\quad \cat_\sA \qquad\quad\forall s\in\cT_\sA\ \lambda^n\text{-a.e.}
\end{equation}
and the interior of the set where (\ref{eq:gaussianbound}) holds is dense in $\cT_\sA$.
\end{thm}
\begin{proof}
From Sard's Theorem \cite{sard42} we know that the set of singular values is of $n$-dimensional Lebesgue measure zero and Theorem \ref{thm:GaussianMixtureFromRegMeasure} applies to the regular values, i.e., moment sequences s.t.\ all representing measures $(C,X)$ have full rank $DS_{k,\sA}(C,X)$.
\end{proof}

The open problem is: Can we ensure that any moment sequence has a representing measure $(C,X)$ with full rank $DS_{k,\sA}(C,X)$ with at most $\cat_\sA$ atoms. If we allow more atoms, this is true by \cite[Lem.\ 36]{didio17w+v+}. But this raises the Carath\'eodory bound. Let us give an example of Theorem \ref{thm:DenseGaussianMixture}.

\begin{exm}\label{exm:uniGaps}
Let $\sA = \{1,x^2,x^3,x^5,x^6\}$ on $\rset$, see \cite[Exm.\ 46]{didio17Cara}. There we found that $\cat_\sA=3$.  It is easily seen that $\sA$ fulfills all condition in Theorem \ref{thm:DenseGaussianMixture} (resp.\ Theorem \ref{thm:GaussianMixtureFromRegMeasure}) and therefore
$\cat_\sA^M(s)\leq 3$ for $s\in\cT_\sA$ $\lambda^n$-a.e.
and the lower bound
$\left\lceil\frac{5}{3}\right\rceil = 2 \leq \cat_\sA^M$
holds because of Theorem \ref{thm:lowerBoundMixture}.
\end{exm}

We end this study with the general one-dimensional result applied to log-normal mixtures. Is uses the fact that $x\in (0,\infty)$ and therefore a prior one-dimensional result \cite[Lem.\ 40]{didio17Cara} can be applied.

\begin{thm}\label{thm:one-dimGapsPos}
Let $m\in\nset$ and $d_1,\dots,d_m\in\nset_0$ be such that $d_1 < \dots < d_m$ and $\sA = \{x^{d_1},\dots,x^{d_m}\}$ on $\cX = (0,\infty)$. Then $\cat_\sA = \left\lceil\frac{m}{2}\right\rceil$.
\end{thm}
\begin{proof}
That $\cat_\sA \geq \left\lceil\frac{m}{2}\right\rceil$ follows from Theorem \ref{thm:LowerBoundCara}.

For $\cat_\sA \leq \left\lceil\frac{m}{2}\right\rceil$ let $s\in\cS_\sA$. By Richter's Theorem (Theorem \ref{thm:richter}) there is a $k$-atomic representing measure ($k\leq m$): $s = \sum_{i=1}^k c_i\cdot s_\sA(x_i) = \sum_{i=1}^k \frac{c_i}{x_i} s_{x\sA}(x_i)$ with $x_i\in (0,\infty)$, where $x\sA = \{x^{d_1+1},\dots,x^{d_m+1}\}$. Hence, w.l.o.g.\ we can assume $d_1>0$.

Let $d_1 > 0$ and we treat the extended homogeneous system $\sB = \{y^{d_{m+1}},x^{d_1} y^{d_{m+1}-d_1}$, $\dots, x^{d_{m+1}}\}$ with $d_{m+1}=2d_m$ on $\overline{\cX}$, i.e., $(x,y)=(0,1)$ is $x=0$ on $\rset$ and $(x,y)=(1,0)$ is $\infty$ on $\rset$. Then $\cS_\sB$ on $\overline{\cX}$ is closed and pointed by \cite[Prop.\ 8]{didio17Cara}. Set $\overline{s} = (s_i)_{i=0}^{m+1} = \sum_{i=1}^k c_i\cdot s_\sB((x_i,1))$, we added the moments $s_0$ and $s_{m+1}$ to $s = (s_i)_{i=1}^m$. Again, by \cite[Prop.\ 8]{didio17Cara} we have that
\begin{align*}
\overline{s}' := \overline{s} &- s_\sB(1,0)\cdot \sup\{c\in\rset \,|\, (\overline{s}-c\cdot s_\sB(0,1))\in\cS_\sB\}\\
& - s_\sB(0,1)\cdot \sup\{c\in\rset \,|\, (\overline{s}-c\cdot s_\sB(1,0))\in\cS_\sB\}\in\partial\cS_\sB
\end{align*}
is a boundary moment sequence of $\cS_\sB$ and by construction of $\overline{s}'$ every representing measure of $\overline{s}'$ does neither contain $(1,0)$ nor $(0,1)$: $\overline{s}' = \sum_{i=1}^k c_i'\cdot s_\sB(x_i',1)$ with $0 < x_1' < \dots < x_k' < \infty$. Since $\overline{s}'$ is a boundary point, $DS_{k,\sB}((c_1',\dots,c_k'),(x_1',\dots,x_k'))$ is singular and from \cite[Lem.\ 40/43]{didio17Cara} we get $k \leq \left\lceil\frac{m}{2}\right\rceil < \left\lceil\frac{m+2}{2}\right\rceil$. But since $s_\sB(0,1) = (1,0,\dots,0)$ and $s_\sB(1,0) = (0,\dots,0,1)$ the $s_i$ for $i=1,\dots,m$ in $\overline{s}$, $\overline{s}'$, and $s$ are not altered, i.e., $s$ has the $k$-atomic representing measure $((c_1',\dots,c_k'),(x_1',\dots,x_m'))$ with $0 < x_1' < \dots < x_k' < \infty$ and $k \leq \left\lceil\frac{m}{2}\right\rceil$.
\end{proof}

\begin{thm}\label{thm:logNormalGapsCaraBound}
Let $m\in\nset$ and $d_1,\dots,d_m\in\nset_0$ be such that $d_1 < \dots < d_m$ and $\sA = \{x^{d_1},\dots,x^{d_m}\}\subset\rset[x]$. Then for the log-normal distribution we have
\[\left\lceil\frac{m}{3}\right\rceil\quad \leq\quad \cat_\sA^M\quad \leq\quad \left\lceil\frac{m}{2}\right\rceil.\]
\end{thm}
\begin{proof}
The lower bound follows from Theorem \ref{thm:lowerBoundMixture}. For the upper bound let $s\in\cT_\sA$. By Theorem \ref{thm:approximateCone}(vi) there is a $\sigma>0$ such that $s\in\cT_{\sA,\sigma}$, i.e., we are in the one-dimensional setup of Theorem \ref{thm:one-dimGapsPos} which gives $\cat_\sA^M(s) \leq \left\lceil\frac{m}{2}\right\rceil$.
\end{proof}

\begin{exm}
Let $\sA = \{1, x, x^2, x^{17}, x^{1863}, x^{25\,376}\}$. Then by using Theorem \ref{thm:logNormalGapsCaraBound} we find that every moment sequence from a log-normal mixture has another log-normal mixture representation with at most $3$ components.
\end{exm}

\section*{Acknowledgment}

The author is grateful to Prof.\ K.\ Schm\"udgen for valuable discussions on the subject of this paper. The author was supported by the Deutsche Forschungsgemeinschaft (SCHM1009/6-1).


\newcommand{\etalchar}[1]{$^{#1}$}

\end{document}